\documentclass[11pt,reqno]{amsart}
\usepackage{color, amsmath,amssymb, amsfonts, amstext,amsthm, latexsym}
\allowdisplaybreaks
\numberwithin{equation}{section}
\newtheorem{theorem}{Theorem}[section]
\newtheorem{definition}[theorem]{Definition}
\newtheorem{lemma}[theorem]{Lemma}
\newtheorem{remark}[theorem]{Remark}
\newtheorem{proposition}[theorem]{Proposition}

\begin{document}

\title[An Integral equation driven by fractional Brownian motion] 
{An integrodifferential equation driven by fractional Brownian motion} 

\author[H. Bessaih] 
{Hakima Bessaih} 

\address{
H. Bessaih,  University of Wyoming,
Department of Mathematics,
Dept. 3036,
1000 East University Avenue,
Laramie WY 82071,
United States}
\email{Bessaih@uwyo.edu} 

\author[C. Wijeratne] 
{Chandana Wijeratne} 
\address{
C. Wijeratne, University of Wyoming,
Department of Mathematics,
Dept. 3036,
1000 East University Avenue,
Laramie WY 82071,
United States}
\email{cwijerat@uwyo.edu} 



\begin{abstract}
This paper deals with the well posedness of an integrodifferential equation that describes a vortex filament associated to a 3D turbulent fluid flow. This equation is driven by a fractional Brownian motion of Hurst parameter $H>1/2$. We prove the global existence and uniqueness of a solution in a functional space of Sobolev type.
\end{abstract}

\maketitle

\section{Introduction} 
The definition of stochastic integrals with respect to the fractional Brownian motions has been investigated intensively by several authors, see for example \cite{AlMaNu01}, \cite{NuRa01}, \cite{CaCoMo03}, \cite{KlZa99} and for a more comprehensive introduction to this topic see \cite{BiHuOkZh08} and \cite{Coutin2007}.  There are two approaches to the construction of such integrals, the pathwise approach that is based on the Riemann-Stieltjes construction and is due to Young \cite{Yo36} and the rough path approach. While the pathwise theory is fairly well understood, it is applicable only for Hurst parameters $H>1/2$, the rough path theory is receiving a lot of interest, see for example \cite{Lyons98}, \cite{GuLeTi06}, \cite{CoLe07} and the references therein and is applicable for Hurst parameters $H>1/4$.

In the present paper, we are using the pathwise argument to solve an integral equation which is an approximation of the line vortex equation. In particular, we will assume that the vorticity field associated to an ideal inviscid incompressible homogeneous fluid in $\mathbb{R}^{3}$ is described by a fractional Brownian motion with Hurst parameter $H>1/2$ and we will study its evolution through a pathline equation. Let us denote by $\omega$ this vorticity field then
$$\vec{\omega}:=\nabla\times \vec{u},$$
where \noindent $\vec{u}$ is the velocity field in $\mathbb{R}^{3}$.
If we denote by $\vec{X}_{t}(\vec{x})$ the position at time $t$ of the fluid particle that at time 0 was at $\vec{x}\in\mathbb{R}^{3}$. We have the following path-lines equation

\begin{equation}\label{path_line}
\frac{d\vec{X}_{t}(\vec{x})}{dt}=\vec{u}(t, \vec{X}_{t}(\vec{x})).
\end{equation}
Now, let us assume that the vorticity field is concentrated on a fractional Brownian curve $\vec{B}^{H}$ as follows

\begin{equation}\label{vortex}
\vec{\omega}(t,{\vec{x}})=\Gamma\int_{0}^{1}\delta({\vec{x}}-\vec{B}^{H}(t,\xi))d\vec{B}^{H}(t,\xi),
\end{equation}
\noindent where $\delta$ is the usual ``Dirac delta function", $\Gamma>0$ is the intensity of vorticity, $\xi\in[0,1]$ is the arc-length, while the parameter $t$ represents the time.
Using the Biot-Savart formula, the equation \eqref{path_line} becomes

\begin{equation}\label{vortex_smooth}
\frac{d\vec{X}(t,\xi)}{dt} = \int_{0}^{1}Q(\vec{X}(t,\xi)-\vec{B}^{H}(t,\eta))d\vec{B}^{H}(t,\eta),
\end{equation}
\noindent with
\begin{equation*}
\vec{X}(0,\xi)=\vec{\phi}(\xi).
\end{equation*}

\noindent Here $\vec{\phi}$ is the initial condition and the matrix valued function $Q$ is the singular matrix\\
\begin{center}
$\frac{-\Gamma}{4\pi|\vec{y}|^{3}}\left(
                         \begin{array}{ccc}
                           0 & y_{3} & -y_{2} \\
                           -y_{3} & 0& y_{1} \\
                           y_{2} & -y_{1} & 0 \\
                         \end{array}
                       \right).$\\
\end{center}
For an introduction to this topic we refer to \cite{BeMa02}, \cite{Ch04}, \cite{Sa92}, \cite{MaPu94} and for a more probabilistic approach to
\cite{Fl02}.  When $\vec{B}^{H}$ is replaced by $\vec{X}$, the equation \eqref{vortex_smooth} has been studied by \cite{BeBe02} for a smooth closed curve $\vec{X}$ in the Sobolev space $W^{1,2}$ and an existence and uniqueness theorem has been proved for local solutions in time. Later, local solutions in some spaces of H\"{o}lder continuous functions have been investigated in \cite{BeGuRu05} that have been extended to global solutions in \cite{BeGu07}.

In the present paper, we will be dealing with an approximation of equation \eqref{vortex_smooth}, the approximation will be on the matrix $Q$,  while the study of equation \eqref{vortex_smooth} is left for a subsequent paper. Furthermore, in order to make the exposition of our results easier to understand, we will make the assumption that the function $\vec{X}$ is a real valued function, however, our results will still be true for a vector valued function $\vec{X}$. More precisely, we are interested in the following integrodifferential equation

\begin{equation}\label{integral_equation}
Y(t,\xi)=\phi(\xi)+\int_{0}^{t}\int_{0}^{\xi}A(Y(s,\eta))dB^{H}(s,\eta)ds,
\end{equation}

\noindent where for all $t\in [0,T],\ B^{H}(t)=\left\{B^{H}(t,\xi),\ \xi\in [0,1]\right\}$ is a real valued fractional Brownian motion of Hurst parameter $H>\frac{1}{2}$, $A$ is a bounded and differentiable real valued function with a Lipschitz continuity property, $\xi$ is a parameter in [0,1].

Let us describe the content of the paper. In Section 2, we introduce the notions of FBM, our assumptions and the functional setting of our problem. In Section 3, we recall the notions of fractional integrals and some related  a priori estimates that would be used later. In Section 4, we introduce the notion of integration wrt to a FBM. Section 5 contains our main results about the existence and uniqueness of a global solution for the equation
\eqref{integral_equation} with their proofs.

\section{Some preliminaries}

\subsection{Fractional Brownian motion}
\begin{definition}
Let $B^{H} = \{B^{H}_{\eta},\eta\geq0\}$ be a stochastic process, and $H\in(0,1)$. $B^{H}$ is called a Fractional Brownian Motion (FBM) with Hurst parameter $H$, if it is a centered Gaussian process with the covariance function
\begin{equation}\label{2.1}
R_{H}(\gamma,\eta)=E[B^{H}(\gamma)B^{H}(\eta)]=\frac{1}{2}(\eta^{2H}+\gamma^{2H}-|\eta-\gamma|^{2H}).
\end{equation}
\end{definition}


\subsection{The integrodifferential equation}
In this paper we study the following equation.

\begin{equation}\label{PDE}
\frac{\partial Y(t,\xi)}{\partial t} = \int_{0}^{\xi}A(Y(t,\eta))dB^{H}(t,\eta),\ \xi\in[0,1]\ {\rm and}\ t\in [0,T]
\end{equation}
\noindent with
\begin{equation*}
Y(0,\xi)=\phi(\xi)
\end{equation*}

\noindent or alternatively, we can consider the integral form

\begin{equation}\label{4.2}
Y(t,\xi) = \phi(\xi) + \int_{0}^{t}\int_{0}^{\xi}A(Y(s,\eta))dB^{H}(s,\eta)ds,
\end{equation}

\noindent where $B^{H}$ is a fractional Brownian motion defined on a complete probability space $(\Omega, \mathcal{F}, P)$, $\phi(\xi)$ is the initial condition and $A:\mathbb{R}\rightarrow \mathbb{R}$ is a measurable function that satisfies the assumptions given below.\\

\subsection{Assumptions}
\noindent Let us assume that:\\
\noindent \textbf{A1}. $A$ is differentiable.\\
\noindent \textbf{A2}. There exists $M_{1}>0$ such that $|A(x)-A(y)| \leq M_{1}|x-y|$ for all $x,y\in\mathbb{R}$.\\
\noindent \textbf{A3}. There exists $M_{2}>0$ such that $|A(x)|\leq M_{2}$ for all $x\in\mathbb{R}$.\\
\noindent \textbf{A4}.  For every $N$ there exists $M_{N}>0$, such that\\$|A'(x)-A'(y)| \leq M_{N}|x-y|$ for all $|x|,|y|\leq N$.

\subsection{Functional Setting}

\noindent Let $\frac{1}{2}<H <1$, $1-H <\alpha <\frac{1}{2}$. We will introduce the following functional spaces.\\

\noindent Let $C([0,T],W^{\alpha,\infty}[0,1])$ be the space of measurable functions
$f:[0,T]\times[0,1]\rightarrow \mathbb{R}$ such that
\begin{equation}
\parallel f\parallel_{\alpha, \infty} := \sup_{t\in[0,T]} \sup_{\xi\in[0,1]} \left(|f(t,\xi)| + \int_{0}^{\xi}\frac{|f(t,\xi)-f(t,\eta)|}{(\xi-\eta)^{\alpha+1}}d\eta\right) < \infty.
\end{equation}\\

\noindent Let $C([0,T],W_{0}^{1-\alpha,\infty}[0,1])$ be the space of measurable functions
$f:[0,T]\times[0,1]\rightarrow \mathbb{R}$ such that
\begin{equation}
\parallel f\parallel_{1-\alpha, \infty,0} := \sup_{t\in[0,T]} \sup_{0<\eta<\xi<1} \left(\frac{|f(t,\xi)-f(t,\eta)|}{(\xi-\eta)^{1-\alpha}} + \int_{\eta}^{\xi}\frac{|f(t,\gamma)-f(t,\eta)|}{(\gamma-\eta)^{2-\alpha}}d\gamma\right) < \infty.
\end{equation}\\

\noindent Let $W^{\alpha,1}([0,1])$ be the space of measurable functions $f:[0,1]\rightarrow \mathbb{R}$ such that
\begin{equation}
\parallel f\parallel_{\alpha, 1} := \int_{0}^{1}\frac{|f(\eta)|}{\eta^{\alpha}}d\eta + \int_{0}^{1}\int_{0}^{\eta}\frac{|f(\eta)-f(\delta)|}{(\eta-\delta)^{\alpha+1}}d\delta d\eta < \infty.
\end{equation}\\

\section{Some a priori estimates}

Since FBM with Hurst parameters $H>1/2$ have sample paths that are $\lambda$-H\"older continuous for all
$\lambda\in (0,H)$, the construction of the integral with respect to a FBM will be performed using a pathwise argument by means of fractional derivatives and integrals. We refer to \cite{KiMaSa93} and \cite{NuRa01} for more details.

\subsection{Fractional integrals and derivatives}

\noindent As usual we denote by $L^{p}(a,b)$ the space of all Lebesgue measurable functions $f:(a,b)\rightarrow \mathbb{R}$ such that
\begin{equation}
\parallel f\parallel_{L^{p}(a,b)} := \Big(\int_{a}^{b}|f(x)|^{p}dx\Big)^{\frac{1}{p}}<\infty
\end{equation}
\noindent for $a<b$ and $1\leq p<\infty$. Let us recall some definitions on Riemann-Liouville fractional integrals and Weyl derivative.

\begin{definition}
Let $f \in L^{1}(a,b)$ and $\alpha > 0$. The left-sided and right-sided Riemann-Liouville fractional integrals of $f$ of order $\alpha$ are defined for almost all $x\in (a,b)$ by

\begin{equation}
I^{\alpha}_{a+}f(x)= \frac{1}{\Gamma(\alpha)} \int_{a}^{x}\frac{f(y)}{(x-y)^{1-\alpha}}dy
\end{equation}

\noindent and

\begin{equation}
I^{\alpha}_{b-}f(x)= \frac{(-1)^{-\alpha}}{\Gamma(\alpha)} \int_{x}^{b}\frac{f(y)}{(y-x)^{1-\alpha}}dy
\end{equation}

\noindent respectively, where $(-1)^{-\alpha} = e^{-i\pi\alpha}$ and $\Gamma(\alpha) = \int_{0}^{\infty}r^{(\alpha -1)}e^{-r}dr$ is the Gamma function or the Euler integral of the second kind.
\end{definition}

\begin{definition}
Suppose $I_{a+}^{\alpha}(L^{p})$ is the image of $L^{p}(a,b)$ under the operator $I_{a+}^{\alpha}$ and $I_{b-}^{\alpha}(L^{p})$ is the image of $L^{p}(a,b)$ under the operator $I_{b-}^{\alpha}$. Let $0<\alpha<1$, then we define the Weyl derivative for almost all $x\in(a,b)$ as

\begin{equation}\label{D plus}
D^{\alpha}_{a+}f(x) = \frac{1}{\Gamma(1-\alpha)}\left(\frac{f(x)}{(x-a)^{\alpha}}
+\alpha\int_{a}^{x}\frac{f(x)-f(y)}{(x-y)^{\alpha+1}}dy\right)1_{(a,b)}(x)
\end{equation}

\noindent when $f\in I_{a+}^{\alpha}(L^{p})$, and

\begin{equation}\label{D minus}
D^{\alpha}_{b-}f(x) = \frac{(-1)^{\alpha}}{\Gamma(1-\alpha)}\left(\frac{f(x)}{(b-x)^{\alpha}}
+\alpha\int_{x}^{b}\frac{f(x)-f(y)}{(y-x)^{\alpha+1}}dy\right)1_{(a,b)}(x)
\end{equation}

\noindent when $f\in I_{b-}^{\alpha}(L^{p})$.

\noindent The convergence of the integrals at the singularity $y=x$ holds pointwise for almost all $x\in(a,b)$ when $p=1$, and in $L^{p}$ sense when $1<p<\infty$.
\end{definition}

\noindent We now introduce the following notations in order to define the generalized Stieltjes integrals. Assuming the limits exist and are finite, let\\

\noindent $f(a+) = \lim_{\epsilon \searrow 0} f(a+\epsilon)$,\\
\noindent $g(b-) = \lim_{\epsilon \searrow 0} g(b-\epsilon)$,\\
\noindent $f_{a+}(x) = [f(x)-f(a+)]1_{(a,b)}(x)$,\\
\noindent $g_{b-}(x) = [g(x)-g(b-)]1_{(a,b)}(x)$.\\

\begin{definition}
Suppose that $f$ and $g$ are functions such that $f(a+)$, $g(a+)$ and $g(b-)$ exist, $f_{a+} \in I^{\alpha}_{a+}(L^{p})$ and $g_{b-} \in I^{1-\alpha}_{b-}(L^{q})$ for some $p$, $q \geq 1$, $\frac{1}{p}+\frac{1}{q} \leq 1$ and  $0<\alpha<1$. Then the generalized Stieltjes integral of $f$ with respect to $g$ is defined as follows, (using \eqref{D plus} and \eqref{D minus}),
\begin{equation}\label{3.4.1}
\int_{a}^{b}fdg = (-1)^{\alpha} \int_{a}^{b}D^{\alpha}_{a+}f_{a+}(x)D^{1-\alpha}_{b-}g_{b-}(x)dx + f(a+)[g(b-)-g(a+)].
\end{equation}
\end{definition}

\noindent Remark that if $\alpha p <1$, under the assumptions of the above definition, we have $f \in I^{\alpha}_{a+}(L^{p})$, and \eqref{3.4.1} can be written as
\begin{equation}\label{3.4.2}
\int_{a}^{b}fdg = (-1)^{\alpha} \int_{a}^{b}D^{\alpha}_{a+}f_{a+}(x)D^{1-\alpha}_{b-}g_{b-}(x)dx.
\end{equation}

\noindent For $a\leq c < d \leq b$, the restriction of $f \in I^{\alpha}_{a+}(L^{p}(a,b))$ to $(c,d)$ belongs to $I^{\alpha}_{c+}(L^{p}(c,d))$ and the continuation of $f \in I^{\alpha}_{c+}(L^{p}(c,d))$ by zero beyond $(c,d)$ belongs to $I^{\alpha}_{a+}(L^{p}(a,b))$. Thus, if $f \in I^{\alpha}_{a+}(L^{p})$ and $g_{b-} \in I^{1-\alpha}_{b-}(L^{q})$, then the integral $\int_{a}^{b}1_{(c,d)}fdg$ in the sense of \eqref{3.4.2} exists for any $a\leq c<d\leq b$, and whenever the left-hand side is defined in the sense of \eqref{3.4.2}, we have
\begin{equation}\label{3.4.3}
\int_{c}^{d}fdg = \int_{a}^{b}1_{(c,d)}fdg.
\end{equation}

\subsection{A priori estimates}

\noindent We have the following basic estimates. Let

\begin{equation}
\Lambda_{\alpha}(g) := \frac{1}{\Gamma(1-\alpha)}\sup_{0<\eta<\xi<1}|(D^{1-\alpha}_{\xi-}g_{\xi-})(\eta)|.
\end{equation}

\noindent Then
\begin{equation}
\Lambda_{\alpha}(g) \leq \frac{1}{\Gamma(1-\alpha)\Gamma(\alpha)}\parallel g \parallel_{1-\alpha, \infty,0} <\infty.
\end{equation}\\

\noindent If $f \in W^{\alpha, 1}(0,1)$, and $g \in W_{0}^{1-\alpha, \infty}(0,1)$ then the integral $\int_{0}^{\xi}fdg$ exists for all $\xi \in [0,1]$.

\noindent Also, by \eqref{3.4.3} we have
\begin{equation*}
\int_{0}^{\xi}fdg = \int_{0}^{1}f1_{(0,\xi)}dg.
\end{equation*}

\noindent Using \eqref{3.4.2} we get
\begin{equation*}
\int_{0}^{\xi}fdg = (-1)^{\alpha}\int_{0}^{\xi}D^{\alpha}_{0+}f(\eta)D_{\xi-}^{1-\alpha}g_{\xi-}(\eta)d\eta.
\end{equation*}
\noindent Then
\begin{equation*}
\left|\int_{0}^{\xi}fdg\right| \leq \sup_{0<\eta<\xi}|D_{\xi-}^{1-\alpha}g_{\xi-}(\eta)|\int_{0}^{\xi}|D_{0+}^{\alpha}f(\eta)|d\eta.
\end{equation*}
\noindent Hence
\begin{equation}\label{3.5.7}
\left|\int_{0}^{\xi}fdg\right| \leq \Lambda_{\alpha}(g)\parallel f\parallel_{\alpha, 1}.
\end{equation}\\

\section{Stochastic integrals with respect to a FBM}

Let us recall the following result, for details and proofs see \cite{NuRa01}.

\begin{lemma}\label{expectation_lemma}
Let $\{B^{H}_{\eta}: \eta\geq0\}$ be a real-valued FBM of Hurst parameter $H\in(\frac{1}{2},1)$. If $1-H<\alpha<\frac{1}{2}$, then
\begin{equation}
\mathbb{E}\sup_{0\leq \gamma \leq \eta \leq 1}|D_{\eta-}^{1-\alpha}B^{H}_{\eta-}(\gamma)|^{p}<\infty,
\end{equation}
\noindent for all $p \in [1,\infty)$.
\end{lemma}

\noindent Let $\{B^{H}_{\eta}: \eta \in [0,1]\}$ be a real-valued FBM, with the Hurst parameter $\frac{1}{2}<H<1$, defined on a complete probability space $(\Omega, \mathcal{F}, P)$. By \eqref{2.1}, we have
\begin{center}
$\mathbb{E}(|B^{H}_{\eta}-B^{H}_{\gamma}|^{2}) = |\eta-\gamma|^{2H}$,
\end{center}
\noindent and for any $p\geq 1$ we have
\begin{equation}
\parallel B^{H}_{\eta}-B^{H}_{\gamma}\parallel_{p} = [\mathbb{E}(|B^{H}_{\eta}-B^{H}_{\gamma}|^{p})]^{\frac{1}{p}} = c_{p}|\eta-\gamma|^{H}.
\end{equation}
\noindent By Lemma \eqref{expectation_lemma} we know that the random variable
\begin{equation}
G = \frac{1}{\Gamma(1-\alpha)}\sup_{0<\gamma<\eta<1}|D_{\eta-}^{1-\alpha}B^{H}_{\eta-}(\gamma)|
\end{equation}
\noindent has moments of all orders.\\

\noindent As a consequence for $1-H<\alpha<\frac{1}{2}$, the pathwise integral $\int_{0}^{\eta}u_{\gamma}dB^{H}_{\gamma}$ exists when $u = \{u_{\eta}, \eta\in[0,1]\}$ is a stochastic process whose trajectories belong to the space $W^{\alpha,1}$ and $B^{H}_{\gamma}$ is a FBM with $H>\frac{1}{2}$. Moreover, we have the estimate
\begin{align}
\left|\int_{0}^{1}u_{\gamma}dB^{H}_{\gamma}\right|\leq G \|u\|_{\alpha,1}.
\end{align}

%

\section{Main results with proofs}

\subsection{Main result}

We state the main result of the present paper:

\begin{theorem}\label{main_theorem}
Let $\alpha\in(1-H,\frac{1}{2})$. Assume that $\phi \in W^{\alpha,\infty}[0,1]$ and the function $A$ satisfies the assumptions A1, A2, A3, and A4. Then for every $T>0$, there exists a unique stochastic process $Y\in L^{0}((\Omega,\mathcal{F},P),C([0,T],W^{\alpha,\infty}[0,1]))$ solution of the equation \eqref{PDE}.
\end{theorem}
As we already defined in Section 4, the stochastic integral wrt to the FBM is well defined pathwise. The above theorem will be proved using a contraction principle that will give us the existence and uniqueness of local solutions. In order to get the global solution, we will need to have an a priori estimate of the solution in some functional spaces. The proof will follow after several steps and will be stated in Section \ref{main_proof}.


\subsection{Fixed point argument}
\noindent Consider the operator
\begin{equation*}
F:C([0,T],W^{\alpha,\infty}[0,1])\longrightarrow C([0,T],W^{\alpha,\infty}[0,1])
\end{equation*}
defined by
\begin{equation}\label{operator}
F(Y(t,\xi)):=\phi(\xi) + \int_{0}^{t}\int_{0}^{\xi}A(Y(s,\gamma))dg_{\gamma}ds,
\end{equation}
\noindent where
$g \in C([0,T],W_{0}^{1-\alpha,\infty}[0,1])$, $\phi \in W^{\alpha,\infty}[0,1]$  and $A$ satisfies the assumptions A1, A2, A3, and A4.
\begin{remark}
Let us remark that the function $g$ in the equation \eqref{operator} is a function of time $t$ and $\gamma$ and the fractional integration is with respect to the parameter $\gamma\in (0,1)$.
\end{remark}

For a given $R>0$, let us define the ball $B_{R, T}$ in $C([0,T],W^{\alpha,\infty}[0,1])$ as
\begin{center}
$B_{R, T} = \{Y\in C([0,T],W^{\alpha,\infty}[0,1]) : ||Y||_{\alpha,\infty} \leq R\}$.
\end{center}


\begin{lemma}\label{lemma1}
Given a positive constant $R_{1}>\|\phi\|_{\alpha,\infty}$, there exists $T_{1}>0$ such that
$F(B_{R_{1},T_{1}})\subseteq B_{R_{1},T_{1}}$. The time $T_{1}$ depends on $R_{1}$, $\alpha$ and the initial condition
$\|\phi\|_{\alpha,\infty}$.
\end{lemma}

\begin{proof}

\begin{equation}\label{alpha_infty_norm}
||F(Y)||_{\alpha,\infty}=\sup_{t\in[0,T],\xi\in[0,1]}\left(|F(Y(t,\xi))| + \int_{0}^{\xi}\frac{|F(Y(t,\xi))-F(Y(t,\eta))|}{(\xi-\eta)^{\alpha+1}}d\eta\right).
\end{equation}

\noindent Consider the first term inside the parenthesis in \eqref{alpha_infty_norm}. We note that

\begin{eqnarray}
|F(Y(t,\xi))|
&\leq & |\phi(\xi)| + \int_{0}^{t}\left|\int_{0}^{\xi}A(Y(s,\gamma))dg_{\gamma}\right|ds\nonumber
\\ &\leq & |\phi(\xi)| + \int_{0}^{t}\left(\Lambda_{\alpha}(g)||A(Y(s,\cdot))||_{\alpha,1}\right)ds,\nonumber
\end{eqnarray}\\

\noindent where
$\Lambda_{\alpha}(g)=\frac{1}{\Gamma(1-\alpha)}\sup_{0<\eta<\xi<1}|D^{1-\alpha}_{\xi-}g_{\xi-}(s,\gamma)|$.\\

\noindent Further,

\begin{align}
& ||A(Y(s,\cdot))||_{\alpha,1}\nonumber
\\&= \int_{0}^{1}\frac{|A(Y(s,\gamma)|}{\gamma^{\alpha}}d\eta + \int_{0}^{1}\int_{0}^{\gamma}\frac{|A(Y(s,\gamma))-A(Y(s,\delta))|}{(\gamma-\delta)^{\alpha+1}}d\delta d\gamma \nonumber
\\&\leq M_{2}\int_{0}^{1}\gamma^{-\alpha}d\gamma + M_{1}\int_{0}^{1}\int_{0}^{\gamma}\frac{|Y(s,\gamma)-Y(s,\delta)|}{(\gamma-\delta)^{\alpha+1}}d\delta d\gamma\nonumber
\\&\leq \frac{M_{2}}{1-\alpha} + M_{1}\int_{0}^{1}\left(\sup_{s\in[0,T],\gamma\in[0,1]}\int_{0}^{\gamma}\frac{|Y(s,\gamma)-Y(s,\delta)|}{(\gamma-\delta)^{\alpha+1}}d\delta\right)d\gamma \nonumber
\\&\leq \frac{M_{2}}{1-\alpha} + M_{1}||Y||_{\alpha,\infty}. \nonumber
\end{align}

\noindent Thus, we see that
\begin{align}\label{estimate_1}
|F(Y(t,\xi))| \leq |\phi(\xi)| + T\sup_{0\leq t\leq T}\Lambda_{\alpha}(g)\left(\frac{M_{2}}{1-\alpha}+M_{1}||Y||_{\alpha,\infty}\right).
\end{align}

\noindent Now, we consider the second term in \eqref{alpha_infty_norm}.

\begin{align}\label{estimate_2}
& \int_{0}^{\xi}\frac{|F(Y(t,\xi))-F(Y(t,\eta))|}{(\xi-\eta)^{\alpha+1}}d\eta\nonumber
\\&= \int_{0}^{\xi}\frac{1}{(\xi-\eta)^{\alpha+1}}\left|\phi(\xi)-\phi(\eta) + \int_{0}^{t}\int_{\eta}^{\xi}A(Y(s,\gamma))dg_{\gamma}ds\right|d\eta\nonumber
\\&\leq \int_{0}^{\xi}\frac{|\phi(\xi)-\phi(\eta)|}{(\xi-\eta)^{\alpha+1}} + \int_{0}^{\xi}\frac{1}{(\xi-\eta)^{\alpha+1}}\left|\int_{0}^{t}\int_{\eta}^{\xi}A(Y(s,\gamma))dg_{\gamma}ds\right|d\eta \nonumber
\\&= \int_{0}^{\xi}\frac{|\phi(\xi)-\phi(\eta)|}{(\xi-\eta)^{\alpha+1}} + \int_{0}^{t}\left(\int_{0}^{\xi}\frac{1}{(\xi-\eta)^{\alpha+1}}\left|\int_{\eta}^{\xi}A(Y(s,\gamma))dg_{\gamma}\right|d\eta \right)ds.
\end{align}

\noindent Again we consider the second term in \eqref{estimate_2} and get 

\begin{align}\label{estimate_3}
& \int_{0}^{\xi}\frac{1}{(\xi-\eta)^{\alpha+1}}\left|\int_{\eta}^{\xi}A(Y(s,\gamma))dg_{\gamma}ds\right|d\eta \nonumber\\
&\leq  \Lambda_{\alpha}(g) \int_{0}^{\xi}\frac{1}{(\xi-\eta)^{\alpha+1}}\bigg(\int_{\eta}^{\xi}
\frac{|A(Y(s,\gamma))|}{(\gamma-\eta)^{\alpha}}d\gamma\nonumber
\\&+ \int_{\eta}^{\xi}\int_{\eta}^{\gamma}\frac{|A(Y(s,\gamma))-A(Y(s,\delta))|}{(\gamma-\delta)^{\alpha+1}}d\delta d\gamma\bigg)d\eta \nonumber
\\&\leq \Lambda_{\alpha}(g)\int_{0}^{\xi}\frac{1}{(\gamma-\eta)^{\alpha+1}}\int_{\eta}^{\xi}\frac{|A(Y(s,\gamma))|}{(\xi-\eta)^{\alpha}}d\gamma d\eta \nonumber
\\&+\Lambda_{\alpha}(g) M_{1}\int_{0}^{\xi}\frac{1}{(\xi-\eta)^{\alpha+1}}\int_{\eta}^{\xi}\int_{\eta}^{\gamma} \frac{|Y(s,\gamma)-Y(s,\delta)|}{(\gamma-\delta)^{\alpha+1}} d\delta d\gamma d\eta.
\end{align}

\noindent First, we use the Fubini theorem on the first term in \eqref{estimate_3} and then the substitution $\eta=\gamma-(\xi-\gamma)x$. Thus,

\begin{align}\label{estimate_4}
& \int_{0}^{\xi}\frac{1}{(\xi-\eta)^{\alpha+1}}\int_{\eta}^{\xi}\frac{|A(Y(s,\gamma))|}{(\gamma-\eta)^{\alpha}}d\gamma d\eta \nonumber
\\&= \int_{0}^{\xi}\int_{\eta}^{\xi}(\xi-\eta)^{-\alpha-1}(\gamma-\eta)^{-\alpha}|A(Y(s,\gamma))|d\gamma d\eta \nonumber
\\&= \int_{0}^{\xi}\int_{0}^{\gamma}(\xi-\eta)^{-\alpha-1}(\gamma-\eta)^{-\alpha}|A(Y(s,\gamma))|d\eta d\gamma \nonumber
\\&= \int_{0}^{\xi}|A(Y(s,\gamma))|\left(\int_{0}^{\gamma}(\xi-\eta)^{-\alpha-1}(\gamma-\eta)^{-\alpha}d\eta\right)d\gamma \nonumber
\\&= \int_{0}^{\xi}|A(Y(s,\gamma))|\left((\xi-\gamma)^{-2\alpha}\int_{0}^{\frac{\gamma}{\xi-\gamma}}(1+x)^{-\alpha-1}x^{-\alpha}dx\right)d\gamma \nonumber
\\&\leq \int_{0}^{\xi}|A(Y(s,\gamma))|\left((\xi-\gamma)^{-2\alpha}\int_{0}^{\infty}(1+x)^{-\alpha-1}x^{-\alpha}dx\right)d\gamma \nonumber
\\&\leq \frac{M_{2}b^{(1)}_{\alpha}}{1-2\alpha}\xi^{1-2\alpha},
\end{align}

\noindent where $b^{(1)}_{\alpha} = B(2\alpha,1-\alpha)$ and $B(p,q)$ is the beta function given by 
\begin{equation*}
B(p,q)=\int_{0}^{1}\frac{x^{q-1}}{(1+x)^{p+q}}dx.
\end{equation*}

%

\noindent Now consider the second part in \eqref{estimate_3}. We see that

\begin{align}\label{estimate_5}
&\int_{0}^{\xi}\frac{1}{(\xi-\eta)^{\alpha+1}}\int_{\eta}^{\xi}\int_{\eta}^{\gamma} \frac{|Y(s,\gamma)-Y(s,\delta)|}{(\gamma-\delta)^{\alpha+1}}d\delta d\gamma d\eta \nonumber
\\&= \int_{0}^{\xi}\int_{\eta}^{\xi}\int_{\eta}^{\gamma}\frac{|Y(s,\gamma)-Y(s,\delta)|}{(\xi-\eta)^{\alpha+1}(\gamma-\delta)^{\alpha+1}}d\delta d\gamma d\eta \nonumber
\\&= \int_{0}^{\xi}\int_{0}^{\gamma}\int_{\eta}^{\gamma}\frac{|Y(s,\gamma)-Y(s,\delta)|}{(\xi-\eta)^{\alpha+1}(\gamma-\delta)^{\alpha+1}}d\delta d\eta d\gamma \nonumber
\\&= \int_{0}^{\xi}\int_{0}^{\gamma}\int_{0}^{\delta}\frac{|Y(s,\gamma)-Y(s,\delta)|}{(\xi-\eta)^{\alpha+1}(\gamma-\delta)^{\alpha+1}}d\eta d\delta d\gamma \nonumber
\\&= \int_{0}^{\xi}\int_{0}^{\gamma}\frac{|Y(s,\gamma)-Y(s,\delta)|}{(\gamma-\delta)^{\alpha+1}}\left(\int_{0}^{\delta}(\xi-\eta)^{-\alpha-1}d\eta\right) d\delta d\gamma \nonumber
\\&\leq \alpha^{-1}\int_{0}^{\xi}\int_{0}^{\gamma}\frac{|Y(s,\gamma)-Y(s,\delta)|}{(\gamma-\delta)^{\alpha+1}}(\xi-\gamma)^{-\alpha}d\delta d\gamma \nonumber
\\&\leq \alpha^{-1}\int_{0}^{\xi}(\xi-\gamma)^{-\alpha}\int_{0}^{\gamma}\frac{|Y(s,\gamma)-Y(s,\delta)|}{(\gamma-\delta)^{\alpha+1}}d\delta d\gamma \nonumber
\\&\leq \alpha^{-1}\int_{0}^{\xi}(\xi-\gamma)^{-\alpha}\sup_{\gamma\in[0,1],s\in[0,T]} \int_{0}^{\gamma}\frac{|Y(s,\gamma)-Y(s,\delta)|}{(\gamma-\delta)^{\alpha+1}}d\delta d\gamma \nonumber
\\&\leq \frac{||Y||_{\alpha,\infty}}{\alpha(1-\alpha)}\xi^{1-\alpha}.
\end{align}

\noindent Now using the estimates \eqref{estimate_1}, \eqref{estimate_2}, \eqref{estimate_3}, \eqref{estimate_4} and \eqref{estimate_5} in \eqref{alpha_infty_norm}, we get

\begin{align}
& ||F(Y)||_{\alpha,\infty} \nonumber
\\ &\leq \sup_{t\in[0,T],\xi\in[0,1]}\bigg[|\phi(\xi)| + \Lambda_{\alpha}(g)t\left(\frac{M_{2}}{1-\alpha}+M_{1}||Y||_{\alpha,\infty}\right) \nonumber
\\ &+ \int_{0}^{\xi}\frac{|\phi(\xi)-\phi(\eta)|}{(\xi-\eta)^{\alpha+1}}d\eta + \int_{0}^{t}\left(\frac{\Lambda_{\alpha}(g)M_{2}b^{(1)}_{\alpha}}{1-2\alpha}\xi^{1-2\alpha} + \frac{\Lambda_{\alpha}(g)M_{1}||Y||_{\alpha,\infty}}{\alpha(1-\alpha)}\xi^{1-\alpha}\right)ds\bigg]\nonumber
\\ &\leq \sup_{t\in[0,T],\xi\in[0,1]}\bigg[\sup_{\xi\in[0,1]}\left(|\phi(\xi)| + \int_{0}^{\xi}\frac{|\phi(\xi)-\phi(\eta)|}{(\xi-\eta)^{\alpha+1}}d\eta\right)\nonumber
\\&+ \Lambda_{\alpha}(g)t\left(\frac{M_{2}}{1-\alpha}+M_{1}||Y||_{\alpha,\infty}\right) + \frac{\Lambda_{\alpha}(g)M_{2}b^{(1)}_{\alpha}}{1-2\alpha}t\xi^{1-2\alpha} + \frac{\Lambda_{\alpha}(g)M_{1}||Y||_{\alpha,\infty}}{\alpha(1-\alpha)}t\xi^{1-\alpha}\bigg]\nonumber
\\ &\leq \sup_{t\in[0,T],\xi\in[0,1]}\bigg[||\phi||_{\alpha,\infty}+ \Lambda_{\alpha}(g)t\left(\frac{M_{2}}{1-\alpha}+M_{1}||Y||_{\alpha,\infty}\right) \nonumber
\\&+ \frac{\Lambda_{\alpha}(g)M_{2}b^{(1)}_{\alpha}}{1-2\alpha}t\xi^{1-2\alpha} + \frac{\Lambda_{\alpha}(g)M_{1}||Y||_{\alpha,\infty}}{\alpha(1-\alpha)}t\xi^{1-\alpha}\bigg]\nonumber
\\ &\leq ||\phi||_{\alpha,\infty}+ T\sup_{0\leq t\leq T} \Lambda_{\alpha}(g)\left[M_{2}\left(\frac{1}{1-\alpha}+ \frac{b^{(1)}_{\alpha}}{1-2\alpha}\right) + M_{1}\left(1+\frac{1}{\alpha(1-\alpha)}\right)||Y||_{\alpha,\infty}\right]\nonumber
\\ &\leq ||\phi||_{\alpha,\infty} + Tb^{(2)}_{\alpha}[1+||Y||_{\alpha,\infty}],\nonumber
\end{align}

\noindent where $b^{(2)}_{\alpha} = \left[M_{2}\left(\frac{1}{1-\alpha}+ \frac{b^{(1)}_{\alpha}}{1-2\alpha}\right)+M_{1}\left(1+\frac{1}{\alpha(1-\alpha)}\right)\right]\sup_{0\leq t\leq T} \Lambda_{\alpha}(g)$.\\

\noindent It is enough to choose $T_{1} = \frac{R_{1}-||\phi||_{\alpha,\infty}}{b^{(2)}_{\alpha}(1+R_{1})}$ with $R_{1}>||\phi||_{\alpha,\infty}$ to get that $F(B_{R_{1}})\subseteq B_{R_{1}}$. This completes the proof.


\end{proof}
\bigskip

\noindent Now in order to prove that the operator $F$ is a contraction, we will need the following two propositions.
\begin{proposition}\label{proposition1}
Let $f\in C([0,T],W_{0}^{\alpha,\infty}[0,1])$ and $g\in C([0,T], W_{0}^{1-\alpha,\infty}[0,1])$. Then, for all $\xi\in[0,1]$ and $t\in [0,T]$
\begin{align}
& \left|\int_{0}^{t}\int_{0}^{\xi}f(s,\gamma)dg_{\gamma}ds\right| \nonumber
\\&+ \int_{0}^{\xi}(\xi-\eta)^{-\alpha-1}\left|\int_{0}^{t}\int_{0}^{\xi}f(s,\gamma)dg_{\gamma}ds - \int_{0}^{t}\int_{0}^{\eta}f(s,\gamma)dg_{\gamma}ds\right|d\eta \nonumber
\\&\leq \sup_{0\leq t\leq T} \Lambda_{\alpha}(g)b_{\alpha}^{(3)}\int_{0}^{t}\int_{0}^{\xi}[(\xi-\gamma)^{-2\alpha}+\gamma^{-\alpha}]\bigg(|f(s,\gamma)|  \nonumber
\\&+ \int_{0}^{\gamma}\frac{|f(s,\gamma)-f(s,\delta)|}{(\gamma-\delta)^{\alpha+1}}d\delta\bigg)d\gamma ds.
\end{align}
\noindent where $b_{\alpha}^{(3)}$ is a constant which depends on $\alpha$, its  explicit expression is given below.
\end{proposition}

\begin{proof}

\noindent Let $0\leq\eta<\xi\leq 1$. Then,

\begin{align}\label{part1}
& \left|\int_{0}^{t}\int_{0}^{\xi}f(s,\gamma)dg_{\gamma}ds-\int_{0}^{t}\int_{0}^{\eta}f(s,\gamma)dg_{\gamma}ds\right|\nonumber
\\ &= \left|\int_{0}^{t}\int_{\eta}^{\xi}f(s,\gamma)dg_{\gamma}ds\right|\nonumber
\\ &= \int_{0}^{t}\left|\int_{\eta}^{\xi}f(s,\gamma)dg_{\gamma}\right|ds\nonumber
\\ &= \int_{0}^{t}|D^{\alpha}_{\eta+}f(s,\gamma)D^{1-\alpha}_{\xi-}g_{\xi-}(s,\gamma)d\gamma| \nonumber
\\ &\leq \sup_{0\leq t\leq T} \Lambda_{\alpha}(g)\int_{0}^{t}\left(\int_{\eta}^{\xi}\frac{|f(s,\gamma)|}{(\gamma-\eta)^{\alpha}}d\gamma + \alpha\int_{\eta}^{\xi}\int_{\eta}^{\gamma}\frac{|f(s,\gamma)-f(s,\delta)|}{(\gamma-\delta)^{\alpha+1}}d\delta d\gamma\right)ds.
\end{align}

\noindent
We multiply \eqref{part1} by $(\xi-\eta)^{-1-\alpha}$, integrate from $0$ to $\xi$, use Fubini's theorem and we get

\begin{align}\label{part2}
& \int_{0}^{\xi}\frac{1}{(\xi-\eta)^{\alpha+1}}\left|\int_{0}^{t}\int_{0}^{\xi}f(s,\gamma)dg_{\gamma}ds- \int_{0}^{t}\int_{0}^{\eta}f(s,\gamma)dg_{\gamma}ds\right|d\eta\nonumber
\\ &\leq \sup_{0\leq t\leq T}\Lambda_{\alpha}(g)\int_{0}^{t}\int_{0}^{\xi}\frac{1}{(\xi-\eta)^{\alpha+1}} \bigg(\int_{\eta}^{\xi}\frac{|f(s,\gamma)|}{(\gamma-\eta)^{\alpha}}d\gamma \nonumber
\\ & + \alpha\int_{\eta}^{\xi}\int_{\eta}^{\gamma}\frac{|f(s,\gamma)-f(s,\delta)|}{(\gamma-\delta)^{\alpha+1}}d\delta d\gamma\bigg)d\eta ds\nonumber
\\ &\leq \sup_{0\leq t\leq T}\Lambda_{\alpha}(g)\bigg(\int_{0}^{t}\int_{0}^{\xi}\int_{\eta}^{\xi}\frac{|f(s,\gamma)|}{(\xi-\eta)^{\alpha+1}(\gamma-\eta)^{\alpha}}d\gamma d\eta ds \nonumber
\\ & + \alpha\int_{0}^{t}\int_{0}^{\xi}\int_{\eta}^{\xi}\int_{\eta}^{\gamma}\frac{|f(s,\gamma)-f(s,\delta)|}{(\xi-\eta)^{\alpha+1} (\gamma-\delta)^{\alpha+1}} d\delta d\gamma d\eta ds\bigg).
\end{align}

\noindent Consider the first term on the right side of \eqref{part2}:

\begin{align}\label{part3}
& \int_{0}^{\xi}\int_{\eta}^{\xi}(\xi-\eta)^{-\alpha-1}(\gamma-\eta)^{-\alpha}|f(s,\gamma)|d\gamma d\eta \nonumber
\\&= \int_{0}^{\xi}\int_{0}^{\gamma}(\xi-\eta)^{-\alpha-1}(\gamma-\eta)^{-\alpha}|f(s,\gamma)|d\eta d\gamma \nonumber
\\&= \int_{0}^{\xi}|f(s,\gamma)|\left(\int_{0}^{\gamma}(\xi-\eta)^{-\alpha-1}(\gamma-\eta)^{-\alpha}d\eta\right)d\gamma \nonumber
\\& \text{using again the same substitution used in \eqref{estimate_3}, we get }\nonumber
\\&\leq b^{(1)}_{\alpha}\int_{0}^{\xi}|f(s,\gamma)|(\xi-\gamma)^{-2\alpha}d\gamma.
\end{align}


\noindent Now, consider the second term on the right side of \eqref{part2}:

\begin{align}\label{part4}
& \int_{0}^{\xi}\int_{\eta}^{\xi}\int_{\eta}^{\gamma}\frac{|f(s,\gamma)-f(s,\delta)|}{(\xi-\eta)^{\alpha+1}(\gamma-\delta)^{\alpha+1}}d\delta d\gamma d\eta \nonumber
\\&= \int_{0}^{\xi}\int_{0}^{\gamma}\int_{\eta}^{\gamma}\frac{|f(s,\gamma)-f(s,\delta)|}{(\xi-\eta)^{\alpha+1}(\gamma-\delta)^{\alpha+1}}d\delta d\eta d\gamma \nonumber
\\&= \int_{0}^{\xi}\int_{0}^{\gamma}\int_{0}^{\delta}\frac{|f(s,\gamma)-f(s,\delta)|}{(\xi-\eta)^{\alpha+1}(\gamma-\delta)^{\alpha+1}}d\eta d\delta d\gamma \nonumber
\\&= \int_{0}^{\xi}\int_{0}^{\gamma}\frac{|f(s,\gamma)-f(s,\delta)|}{(\gamma-\delta)^{\alpha+1}}\left[\int_{0}^{\delta}(\xi-\eta)^{-\alpha-1}d\eta\right] d\delta d\gamma \nonumber
\\&\leq \alpha^{-1}\int_{0}^{\xi}\int_{0}^{\gamma}\frac{|f(s,\gamma)-f(s,\delta)|}{(\gamma-\delta)^{\alpha+1}}(\xi-\gamma)^{-\alpha}d\delta d\gamma \nonumber
\\&\leq \alpha^{-1}\int_{0}^{\xi}(\xi-\gamma)^{-\alpha}\int_{0}^{\gamma}\frac{|f(s,\gamma)-f(s,\delta)|}{(\gamma-\delta)^{\alpha+1}}d\delta d\gamma.
\end{align}

\noindent Using \eqref{part2} \eqref{part3} and \eqref{part4}, we obtain

\begin{align}\label{part5}
& \int_{0}^{\xi}\frac{1}{(\xi-\eta)^{\alpha+1}}\left|\int_{0}^{t}\int_{0}^{\xi}f(s,\gamma)dg_{\gamma}ds- \int_{0}^{t}\int_{0}^{\eta}f(s,\gamma)dg_{\gamma}ds\right|d\eta\nonumber
\\ & \leq \sup_{0\leq t\leq T}\Lambda_{\alpha}(g)\int_{0}^{t}\bigg(b^{(1)}_{\alpha}\int_{0}^{\xi}|f(s,\gamma)|(\xi-\gamma)^{-2\alpha}d\gamma \nonumber
\\ & + \int_{0}^{\xi}(\xi-\gamma)^{-\alpha}\int_{0}^{\gamma}\frac{|f(s,\gamma)-f(s,\delta)|}{(\gamma-\delta)^{\alpha+1}}d\delta d\gamma\bigg)ds.
\end{align}

\noindent By \eqref{part1} we have

\begin{align}\label{part6}
& \left|\int_{0}^{t}\int_{0}^{\xi}f(s,\gamma)dg_{\gamma}ds\right|\nonumber
\\ & \leq \sup_{0\leq t\leq T}\Lambda_{\alpha}(g)\int_{0}^{t}\left(\int_{0}^{\xi}\frac{|f(s,\gamma)|}{\gamma^{\alpha}}d\gamma + \alpha\int_{0}^{\xi}\int_{0}^{\gamma}\frac{|f(s,\gamma)-f(s,\delta)|}{(\gamma-\delta)^{\alpha+1}}d\delta d\gamma\right)ds.
\end{align}

\noindent From \eqref{part5} and \eqref{part6}, we have

\begin{align}
& \left|\int_{0}^{t}\int_{0}^{\xi}f(s,\gamma)dg_{\gamma}ds\right| +
\int_{0}^{\xi}\frac{1}{(\xi-\eta)^{\alpha+1}}\left|\int_{0}^{t}\int_{0}^{\xi}f(s,\gamma)dg_{\gamma}ds- \int_{0}^{t}\int_{0}^{\eta}f(s,\gamma)dg_{\gamma}ds\right|d\eta\nonumber
\\ & \leq \sup_{0\leq t\leq T}\Lambda_{\alpha}(g)\int_{0}^{t}\bigg(b^{(1)}_{\alpha}\int_{0}^{\xi}\frac{|f(s,\gamma)|}{(\xi-\gamma)^{2\alpha}}d\gamma + \int_{0}^{\xi}\frac{|f(s,\gamma)|}{\gamma^{\alpha}}d\gamma \nonumber
\\ &+ \int_{0}^{\xi}[(\xi-\gamma)^{-\alpha}+\alpha]\int_{0}^{\gamma}\frac{|f(s,\gamma)-f(s,\delta)|}{(\gamma-\delta)^{\alpha+1}}d\delta d\gamma\bigg)ds\nonumber
\\& \text{using the fact that $\alpha<\gamma^{-\alpha}$ and $\frac{1}{(\xi-\gamma)^{\alpha}}<\frac{1}{(\xi-\gamma)^{2\alpha}}$ we get }
\\ & \leq \sup_{0\leq t\leq T}\Lambda_{\alpha}(g)\int_{0}^{t}\int_{0}^{\xi}\bigg([b^{(1)}_{\alpha}(\xi-\gamma)^{-2\alpha}+\gamma^{-\alpha}]|f(s,\gamma)|\nonumber
\\ &+ [(\xi-\gamma)^{-2\alpha}+\gamma^{-\alpha}]\int_{0}^{\gamma}\frac{|f(s,\gamma)-f(s,\delta)|}{(\gamma-\delta)^{\alpha+1}}d\delta\bigg)d\gamma ds\nonumber
\\ & \leq \sup_{0\leq t\leq T}\Lambda_{\alpha}(g)b^{(3)}_{\alpha}\int_{0}^{t}\int_{0}^{\xi}\bigg([(\xi-\gamma)^{-2\alpha}+\gamma^{-\alpha}]|f(s,\gamma)| \nonumber
\\ &+ [(\xi-\gamma)^{-2\alpha}+\gamma^{-\alpha}]\int_{0}^{\gamma}\frac{|f(s,\gamma)-f(s,\delta)|}{(\gamma-\delta)^{\alpha+1}}d\delta\bigg)d\gamma ds\nonumber
\\ & \leq \sup_{0\leq t\leq T}\Lambda_{\alpha}(g)b^{(3)}_{\alpha}\int_{0}^{t}\int_{0}^{\xi}[(\xi-\gamma)^{-2\alpha}+\gamma^{-\alpha}]\bigg(|f(s,\gamma)| + \int_{0}^{\gamma}\frac{|f(s,\gamma)-f(s,\delta)|}{(\gamma-\delta)^{\alpha+1}}d\delta\bigg)d\gamma ds,\nonumber
\end{align}
\noindent where $b^{(3)}_{\alpha}=\max\{1,b^{(1)}_{\alpha}\}$. This completes the proof.
\end{proof}

\bigskip

\begin{proposition}\label{proposition2}
Let $h:\mathbb{R}\rightarrow\mathbb{R}$ be a function satisfying the assumptions A3 and A4. Then for all $N>0$ and $|X_{1}|,|X_{2}|,|X_{3}|,|X_{4}| \leq N$ for all $X_{1},X_{2},X_{3},X_{4}\in\mathbb{R}$,
\begin{align}
& |h(X_{1})-h(X_{2})-h(X_{3})+h(X_{4})|\nonumber
\\ & \leq M_{1}|X_{1}-X_{2}-X_{3}+X_{4}| + M_{N}|X_{1}-X_{3}|(|X_{1}-X_{2}|+|X_{3}-X_{4}|).
\end{align}

\end{proposition}

\begin{proof}

\noindent Using the mean value theorem, we get
\begin{eqnarray}
\frac{h(X_{1})-h(X_{3})}{X_{1}-X_{3}}\nonumber &=& h'(\theta X_{1}+(1-\theta)X_{3})\quad (0<\theta<1)\nonumber
\\ &=& h'(\beta)\quad \text{where} \quad \beta=\theta X_{1}+(1-\theta)X_{3},\quad X_{1}<\beta<X_{3}\nonumber
\\ &=& \frac{1}{X_{3}-X_{1}}\int_{X_{1}}^{X_{3}}h'(\gamma)d\gamma.\nonumber
\end{eqnarray}
\noindent Using the substitution $\gamma =  \theta X_{1}+(1-\theta)X_{3}$, we get
\begin{equation}\label{MVT1}
h(X_{1})-h(X_{3}) = (X_{1}-X_{3})\int_{0}^{1}h'(\theta X_{1}+(1-\theta)X_{3})d\theta.
\end{equation}
\noindent Similarly we get
\begin{equation}\label{MVT2}
h(X_{2})-h(X_{4}) = (X_{2}-X_{4})\int_{0}^{1}h'(\theta X_{2}+(1-\theta)X_{4})d\theta.
\end{equation}
\noindent By \eqref{MVT1} and \eqref{MVT2} we obtain
\begin{align}
& |h(X_{1})-h(X_{2})-h(X_{3})+h(X_{4})|\nonumber
\\ & = \left|(X_{1}-X_{3})\int_{0}^{1}h'(\theta X_{1}+(1-\theta)X_{3})d\theta - (X_{2}-X_{4})\int_{0}^{1}h'(\theta X_{2}+(1-\theta)X_{4})d\theta\right|\nonumber
\\ & = \int_{0}^{1}(X_{1}-X_{2}-X_{3}+X_{4})h'(\theta X_{2}+(1-\theta)X_{4})d\theta\nonumber
\\ & + \int_{0}^{1}(X_{1}-X_{3})[h'(\theta X_{1}+(1-\theta)X_{3}) - h'(\theta X_{2}+(1-\theta)X_{4})]d\theta\nonumber
\\ & \leq M_{1}|X_{1}-X_{2}-X_{3}+X_{4}| \nonumber
\\ & + |X_{1}-X_{3}|\int_{0}^{1}|[h'(\theta X_{1}+(1-\theta)X_{3}) - h'(\theta X_{2}+(1-\theta)X_{3})]d\theta \nonumber
\\ & + |X_{1}-X_{3}|\int_{0}^{1}|[h'(\theta X_{1}+(1-\theta)X_{3}) - h'(\theta X_{2}+(1-\theta)X_{4})]d\theta \nonumber
\\ & \leq M_{1}|X_{1}-X_{2}-X_{3}+X_{4}| + M_{N}|X_{1}-X_{3}|\bigg(|X_{1}-X_{2}|\int_{0}^{1}\theta d\theta \nonumber
\\&+ |X_{3}-X_{4}|\int_{0}^{1}\theta d\theta\bigg)\nonumber
\\ & \leq M_{1}|X_{1}-X_{2}-X_{3}+X_{4}| + M_{N}|X_{1}-X_{3}|(|X_{1}-X_{2}| + |X_{3}-X_{4}|).\nonumber
\end{align}
This completes the proof.
\end{proof}


\begin{lemma}\label{lemma2}
Given a positive constant $R_{2}>\|\phi\|_{\alpha,\infty}$, there exists $T_{2}>0$  and a constant
$0<C<1$ such that,
\begin{equation*}
||F(Y_{1})-F(Y_{2})||_{\alpha,\infty}\leq C||Y_{1}-Y_{2}||_{\alpha,\infty}
\end{equation*}
\noindent for all $Y_{1},Y_{2}\in B_{R_{2},T_{2}}$.
\end{lemma}

\begin{proof}
Recall $F(Y(t,\xi))=\phi(\xi) + \int_{0}^{t}\int_{0}^{\xi}A(Y(s,\gamma))dg_{\gamma}ds$. Then,
\begin{equation*}
F(Y_{1}(t,\xi))-F(Y_{2}(t,\xi))=\int_{0}^{t}\int_{0}^{\xi}[A(Y_{1}(s,\gamma))-A(Y_{2}(s,\gamma))]dg_{\gamma}ds.
\end{equation*}

\noindent Using the expression \eqref{alpha_infty_norm} and Proposition \ref{proposition1}, we have that

\begin{align}
& ||F(Y_{1})-F(Y_{2})||_{\alpha,\infty} \nonumber
\\&\leq \sup_{0\leq t\leq T}\Lambda_{\alpha}(g) b^{(1)}_{\alpha}\sup_{t\in[0,T],\xi\in[0,1]}\int_{0}^{t}\int_{0}^{\xi}[(\xi-\gamma)^{-2\alpha}+\gamma^{-\alpha}]\bigg[|A(Y_{1}(s,\gamma))
-A(Y_{2}(s,\gamma))| \nonumber
\\&+ \int_{0}^{\gamma}\frac{|A(Y_{1}(s,\gamma))
-A(Y_{2}(s,\gamma))-\{A(Y_{1}(s,\eta))
-A(Y_{2}(s,\eta))\}|}{(\gamma-\eta)^{\alpha+1}}d\eta\bigg]d\gamma ds \nonumber
\\&\leq \sup_{0\leq t\leq T}\Lambda_{\alpha}(g) b^{(1)}_{\alpha}\sup_{t\in[0,T],\xi\in[0,1]}\int_{0}^{t}\int_{0}^{\xi}[(\xi-\gamma)^{-2\alpha}+\gamma^{-\alpha}]\bigg[M_{1}|(Y_{1}(s,\gamma))
-(Y_{2}(s,\gamma))| \nonumber
\\&+ \int_{0}^{\gamma}\frac{|A(Y_{1}(s,\gamma))
-A(Y_{2}(s,\gamma))-A(Y_{1}(s,\eta))
+A(Y_{2}(s,\eta))|}{(\gamma-\eta)^{\alpha+1}}d\eta\bigg]d\gamma ds. \nonumber
\end{align}

\noindent By Proposition \ref{proposition2} we get

\begin{align}
& ||F(Y_{1})-F(Y_{2})||_{\alpha,\infty} \nonumber
\\&\leq \sup_{0\leq t\leq T}\Lambda_{\alpha}(g) b^{(1)}_{\alpha}\sup_{t\in[0,T],\xi\in[0,1]}\int_{0}^{t}\int_{0}^{\xi}[(\xi-\gamma)^{-2\alpha}+\gamma^{-\alpha}]
\bigg[M_{1}|Y_{1}(s,\gamma)-Y_{2}(s,\gamma)| \nonumber
\\&+ M_{1}\int_{0}^{\gamma}\frac{|Y_{1}(s,\gamma)-Y_{2}(s,\gamma)-Y_{1}(s,\eta)
+Y_{2}(s,\eta)|}{(\gamma-\eta)^{\alpha+1}}d\eta \nonumber
\\&+M_{N}|Y_{1}(s,\gamma)-Y_{2}(s,\gamma)|\nonumber
\\& \times\bigg(\int_{0}^{\gamma}\frac{|Y_{1}(s,\gamma)-Y_{1}(s,\eta)|}{(\gamma-\eta)^{\alpha+1}}d\eta + \int_{0}^{\gamma}\frac{|Y_{2}(s,\gamma)-Y_{2}(s,\eta)|}{(\gamma-\eta)^{\alpha+1}}d\eta\bigg)\bigg]d\gamma ds. \nonumber
\end{align}

\noindent Using the notation
\begin{equation*}
\Delta(Y_{j})=\sup_{s\in[0,T],\gamma\in[0,1]}\int_{0}^{\gamma}\frac{|Y_{j}(s,\gamma)-Y_{j}(s,\eta)|}{(\gamma-\eta)^{\alpha+1}}d\eta,\quad j\in\{1,2\},
\end{equation*}
\noindent we get

\begin{align}
& ||F(Y_{1})-F(Y_{2})||_{\alpha,\infty} \nonumber
\\&\leq \sup_{0\leq t\leq T}\Lambda_{\alpha}(g) b^{(1)}_{\alpha}\sup_{t\in[0,T],\xi\in[0,1]}\int_{0}^{t}\int_{0}^{\xi}[(\xi-\gamma)^{-2\alpha}+\gamma^{-\alpha}]\bigg[M_{1}|Y_{1}(s,\gamma)
-Y_{2}(s,\gamma)| \nonumber
\\&+ M_{1}\int_{0}^{\gamma}\frac{|Y_{1}(s,\gamma)-Y_{2}(s,\gamma)-Y_{1}(s,\eta)+Y_{2}(s,\eta)|}{(\gamma-\eta)^{\alpha+1}}d\eta \nonumber
\\&+ M_{N}|Y_{1}(s,\gamma)-Y_{2}(s,\gamma)|(\Delta(Y_{1})+\Delta(Y_{2}))\bigg]d\gamma ds \nonumber
\\&\leq \sup_{0\leq t\leq T}\Lambda_{\alpha}(g) b^{(1)}_{\alpha}b^{(4)}\sup_{t\in[0,T],\xi\in[0,1]}\int_{0}^{t}\int_{0}^{\xi}[(\xi-\gamma)^{-2\alpha}+\gamma^{-\alpha}] \bigg[|Y_{1}(s,\gamma)-Y_{2}(s,\gamma)| \nonumber
\\&+ \int_{0}^{\gamma}\frac{|Y_{1}(s,\gamma)-Y_{2}(s,\gamma)-Y_{1}(s,\eta)+Y_{2}(s,\eta)|}{(\gamma-\eta)^{\alpha+1}}d\eta\bigg]d\gamma ds, \nonumber
\end{align}

\noindent where $b^{(4)} = (M_{1}+M_{{R_{1}}})(1+2R_{1})$. Hence,

\begin{align}
& ||F(Y_{1})-F(Y_{2})||_{\alpha,\infty} \nonumber
\\&\leq \sup_{0\leq t\leq T}\Lambda_{\alpha}(g) b^{(1)}_{\alpha}b^{(4)}\sup_{t\in[0,T],\xi\in[0,1]}\int_{0}^{t}\int_{0}^{\xi}[(\xi-\gamma)^{-2\alpha}+\gamma^{-\alpha}]\nonumber
\\&\times \sup_{t\in[0,T],\xi\in[0,1]}\bigg(|Y_{1}(s,\gamma)-Y_{2}(s,\gamma)|\nonumber
\\&\times \int_{0}^{\gamma} \frac{|Y_{1}(s,\gamma)-Y_{2}(s,\gamma)-\{Y_{1}(s,\eta)-Y_{2}(s,\eta)\}|}{(\gamma-\eta)^{\alpha+1}}d\eta\bigg)d\gamma ds \nonumber
\\&= \sup_{0\leq t\leq T}\Lambda_{\alpha}(g) b^{(1)}_{\alpha}b^{(4)}\sup_{t\in[0,T],\xi\in[0,1]}\int_{0}^{t}\int_{0}^{\xi}[(\xi-\gamma)^{-2\alpha}+ \gamma^{-\alpha}]||Y_{1}(s,\gamma)-Y_{2}(s,\gamma)||_{\alpha,\infty}d\gamma ds \nonumber
\\&= \sup_{0\leq t\leq T}\Lambda_{\alpha}(g) b^{(1)}_{\alpha}b^{(4)}||Y_{1}-Y_{2}||_{\alpha,\infty}\sup_{t\in[0,T],\xi\in[0,1]} \left(\frac{\xi^{1-2\alpha}}{1-2\alpha}+\frac{\xi^{1-\alpha}}{1-\alpha}\right)t \nonumber
\\&\leq \sup_{0\leq t\leq T}\Lambda_{\alpha}(g) b^{(1)}_{\alpha}b^{(4)}\frac{2-3\alpha}{(1-2\alpha)(1-\alpha)}T||Y_{1}-Y_{2}||_{\alpha,\infty} \nonumber
\\&= b^{(5)}T||Y_{1}-Y_{2}||_{\alpha,\infty}, \nonumber
\end{align}

\noindent where $b^{(5)} = b^{(1)}_{\alpha}b^{(4)}\frac{2-3\alpha}{(1-2\alpha)(1-\alpha)}$.\\

\noindent Choose $T_{2}$ such that $b^{(5)}T_{2} = C < 1$. This completes the proof.
%
\end{proof}

\noindent Now, we are able to state the following theorem:


\begin{theorem}\label{theorem_local}
Let $0<\alpha<\frac{1}{2}$, $g\in C([0,T], W_{0}^{1-\alpha,\infty}[0,1])$. Consider the integrodifferential equation
\begin{equation}\label{PDEg}
Y(t,\xi) = \phi(\xi) + \int_{0}^{t}\int_{0}^{\xi}A(Y(s,\eta))dg(s,\eta)ds
\end{equation}
\noindent where $t \in [0,T]$, $\xi \in [0,1]$. Assume that $A$ satisfies assumptions A1, A2, A3, and A4 and that
$\phi \in W^{\alpha,\infty}([0,1])$. Then, there exists $T_{0}>0$ such that the above equation has a unique solution
\begin{center}
$Y \in C([0,T],W^{\alpha,\infty}[0,1])$,
\end{center}
for all $T\leq T_{0}$.

\end{theorem}

\begin{proof}  Choose $T_{0} = \min\{T_{1},T_{2}\}$, and $R=\min\{R_{1},R_{2}\}$.
Then, using Lemma \ref{lemma1} and Lemma \ref{lemma2} the operator $F$
is a contraction on $B_{R,T}$ for all $T\leq T_{0}$ and this completes the proof.
\end{proof}

\noindent Now, we can show that the solution of \eqref{PDEg} is global in time.\\

\begin{theorem}\label{theorem_global}
Let $1-H<\alpha<\frac{1}{2}$, $g\in C([0,T],W^{1-\alpha,\infty,0}[0,1])$. Assume that
$A$ satisfies assumptions A1, A2, A3, and A4 and that $\phi \in W^{\alpha,\infty}([0,1])$.
Then for all $T>0$, there exists a unique $Y \in C([0,T],W^{\alpha,\infty}[0,1])$ solution of \eqref{PDEg}.

\end{theorem}
\begin{proof}


%

\noindent It is enough to get an estimate in $C([0,T],W^{\alpha,\infty}[0,1])$. We can write
\begin{equation}
||Y(t)||_{\alpha,\infty} = \sup_{\xi\in[0,1]}\left(Y(t,\xi)+\int_{0}^{\xi}\frac{|Y(t,\xi)-Y(t,\eta)|}{(\xi-\eta)^{\alpha+1}}d\eta\right).
\end{equation}

\noindent Consider the first term on the right side of the above equality,
\begin{equation*}
|Y(t,\xi)| \leq |\phi(\xi)| + \int_{0}^{t}\left|\int_{0}^{\xi}A(Y(s,\eta))dg(s,\eta)\right|ds.
\end{equation*}

\noindent From \eqref{estimate_1} we obtain
\begin{align}\label{estimate_global_1}
|Y(t,\xi)| \leq |\phi(\xi)| + \int_{0}^{t}\Lambda_{\alpha}(g)t\left(\frac{M_{2}}{1-\alpha}+M_{1}||Y(s)||_{\alpha,\infty}ds\right).
\end{align}

\noindent Consider the second term. We have
\begin{align*}
& \left|\int_{0}^{\xi}\frac{|Y(t,\xi)-Y(t,\eta)|}{(\xi-\eta)^{\alpha+1}}d\eta\right| \nonumber
\\& = \int_{0}^{\xi}\frac{1}{(\xi-\eta)^{\alpha+1}}\left|\phi(\xi)+\int_{0}^{t}\int_{0}^{\xi}A(Y(s,\gamma))dg_{\gamma}ds - \left(\phi(\eta)+\int_{0}^{t}\int_{0}^{\eta}A(Y(s,\gamma))dg_{\gamma}ds\right)\right|\nonumber
\\& \leq \int_{0}^{\xi}\frac{1}{(\xi-\eta)^{\alpha+1}}\left(|\phi(\xi)-\phi(\eta)|+\int_{0}^{t}\left|\int_{\eta}^{\xi}A(Y(s,\gamma))dg_{\gamma}\right|ds d\eta \right)\nonumber
\\& \leq \int_{0}^{\xi}\frac{|\phi(\xi)-\phi(\eta)|}{(\xi-\eta)^{\alpha+1}}d\eta + \int_{0}^{t}\int_{0}^{\xi}\frac{1}{(\xi-\eta)^{\alpha+1}}\left|\int_{\eta}^{\xi}A(Y(s,\gamma))dg_{\gamma}\right|d\eta ds\nonumber.
\end{align*}

\noindent From \eqref{estimate_4} and \eqref{estimate_5} we obtain

\begin{align}\label{estimate_global_2}
& \left|\int_{0}^{\xi}\frac{|Y(t,\xi)-Y(t,\eta)|}{(\xi-\eta)^{\alpha+1}}d\eta\right| \nonumber
\\& \leq \int_{0}^{\xi}\frac{|\phi(\xi)-\phi(\eta)|}{(\xi-\eta)^{\alpha+1}}d\eta + \Lambda_{\alpha}(g)\int_{0}^{t}\left(\frac{M_{2}b_{\alpha}^{(1)}}{1-2\alpha}\xi^{1-2\alpha}
+\frac{M_{1}}{\alpha(1-\alpha)}||Y(s)||_{\alpha,\infty}\xi^{1-\alpha}\right)ds.
\end{align}

\noindent By \eqref{estimate_global_1} and \eqref{estimate_global_2} we get,
\begin{align}
& ||Y(t)||_{\alpha,\infty} \nonumber
\\& \leq ||\phi||_{\alpha,\infty}+\Lambda_{\alpha}(g)\int_{0}^{t}\left(M_{2}\left(\frac{1}{1-\alpha}+ \frac{b_{\alpha}^{(1)}}{1-2\alpha}\right)
+M_{1}\left(1+\frac{1}{\alpha(1-\alpha)}\right)||Y(s)||_{\alpha,\infty}\right)ds.\nonumber
\end{align}

\noindent By Gronwall's inequality we obtain,
$$||Y||_{\alpha,\infty}\leq ||\phi||_{\alpha,\infty}\exp(KT),$$


\noindent where $K = \sup_{0\leq t\leq T}\Lambda_{\alpha}(g)\left[M_{2}\left(\frac{1}{1-\alpha}+ \frac{b_{\alpha}^{(1)}}{1-2\alpha}\right)
+M_{1}\left(1+\frac{1}{\alpha(1-\alpha)}\right)\right]$.\\

Hence, the local solution is global in time.
\end{proof}
%

\subsection{Proof of Theorem \ref{main_theorem}}\label{main_proof}

For every $t\in [0,T]$, the random variable $G=\frac{1}{\Gamma(1-\alpha)}\sup_{0<\eta<\xi<1}|(D^{1-\alpha}_{\xi-}B_{\xi-})(t,\eta)|$ has moments of all orders by Proposition \ref{expectation_lemma}. Hence, the pathwise integral $\int_{0}^{1}A(Y(t,\eta))dB_{\eta}(t)$ exists for $1-H<\alpha<\frac{1}{2}$ and the existence and uniqueness of solutions follows from Theorem \ref{theorem_global} which completes the proof.
\bigskip

\noindent {\bf Acknowledgments:} This work was partially
supported by NSF grant No. DMS 0608494.

\end{document}